\documentclass[12pt,reqno]{amsart}
\usepackage{graphicx}
\newtheorem{thm}{Theorem}[section]

\newtheorem{lem}[thm]{Lemma}

\theoremstyle{definition}

\theoremstyle{remark}
\newtheorem{rem}[thm]{Remark}

\numberwithin{equation}{section}

\makeatletter
\newcommand*{\rom}[1]{\expandafter\@slowromancap\romannumeral #1@}
\makeatother
\begin{document}

\title[On convergence to a football]{On convergence to a football}
\author{Hao Fang}
\address{14 MacLean Hall, University of Iowa, Iowa City, IA, 52242}
\email{hao-fang@uiowa.edu}
\thanks{H.F.'s work is partially supported by Simons Foundation and NSF DMS-100829. }
\author{Mijia Lai}%
\address{800 Dongchuan RD, Shanghai Jiao Tong University, Shanghai, China,200240 }%
\email{laimijia@sjtu.edu.cn}%
\thanks{M.L.'s work is partially supported by Shanghai Yangfan 15YF1406200.}

\begin{abstract}
We show that spheres of positive constant curvature with $n$ ($n\geq3$) conic points converge to a sphere of positive constant curvature with two conic points (or called an (American) football) in Gromov-Hausdorff topology when the corresponding singular divisors converge to a critical divisor  in the sense of Troyanov.

We prove this convergence in two different ways. Geometrically, the convergence follows from Luo-Tian's explicit description of conic spheres as boundaries of convex polytopes in $S^{3}$. Analytically, regarding the conformal factors as the singular solutions to the corresponding PDE, we derive the required a priori estimates  and convergence result after proper reparametrization.

\end{abstract}


\maketitle
\section{Introduction}

We study convergence properties of metrics on Riemann surface with conical singularities. It has been an old topic and been first studied as early as in 1905 by Picard~\cite{P}, when he was considering the uniformization problem for Riemann surfaces with branched points. In the 1990s, the prescribing curvature problem on Riemann surface with conic singularities was studied by several authors~\cite{Mc, Tr, LT, CL1, CL2}. In higher dimensions, the K\"{a}hler metric with cone singularities along a divisor was first considered by Tian~\cite{T1}. Recent progress on the study of K\"{a}hler-Einstein metrics on Fano manifolds has brought up renewed interests in such conical metrics~\cite{D}.

We first start with some definitions.

For a compact closed Riemann surface $S$, a metric $g$ is said to have a conic singularity of order $\beta\in(-1,\infty)$ at $p$, if under a local holomorphic coordinate centered at $p$,
\[
g=e^{f(z)}|z|^{2\beta} |dz|^2,
\]
where $f(z)$ is continuous and $C^2$ away from $p$.

The singularity is modeled on the Euclidean cone: ${\mathbb C}$ with the metric $|z|^{2\beta}|dz|^2$ is isometric to a Euclidean cone of cone angle  $2\pi(\beta+1)$. In general, we use the  triple $(S, g ,D)$, where $D=\sum_{i=1}^{n} \beta_i p_i$, to denote
a Riemann surface $S$ endowed with a metric $g$ with conic singularities at each $p_i$ of order $\beta_i$.

Let $K=K(g)$  be the Gaussian curvature of $g$ where it is smooth. If we denote $|D|=\sum_{i=1}^{n} \beta_i$, the Gauss-Bonnet formula for the conic surface $(S, g, D)$ becomes
\[
\int_{S} K ds^2= 2\pi \chi(S,D),
\] where
\[
\chi(S, D):=\chi+|D|
\]  can be interpreted as the Euler characteristic number for the conic surface $(S, D)$.

Troyanov~\cite{Tr} systematically studies the prescribing curvature problem on conic surfaces. He divides the problem, according to the Euler characteristic, to the following  cases:
 \begin{description}
   \item[1. negative case]$\chi(S, D)<0$;
   \item[2. zero case] $\chi(S, D)=0$;
   \item[3. positive case]  $\chi(S, D)>0$; \begin{description}
          \item[3.a subcritical case]$\chi(S, D)< \min\{2, 2+2\min_{i} \beta_i\}$;
          \item[3.b critical case] $\chi(S, D)=\min\{2, 2+2\min_{i} \beta_i\}$;
          \item[3.c supercritical case ]  $\chi(S, D)>\min\{2, 2+2\min_{i} \beta_i\}$.
        \end{description}
 \end{description}
It turns out that cases 1, 2, and 3.a are parallel to the corresponding cases in the prescribing curvature problem on a smooth surface as the corresponding functionals are coercive; while cases 3.b and 3.c are more delicate. We refer the reader to ~\cite{BDM, CL1, E, LZ} for interesting works.

For the Yamabe problem on surfaces with the conic singularity, where we prescribe the constant curvature, the answer is more complete. Without loss of generality, we assume that $S$ is orientable. If $\chi(S,D)\leq 0$, it has been shown~\cite{Tr} that there always admits a conic metric with constant curvature, unique up to scaling. For $\chi(S, D)>0$, $S$ would necessarily be $S^2$, if in addition $\beta_i\in (-1,0)$, then $S$ admits a conic metric of positive constant curvature if and only if :

\begin{itemize}
  \item  $n=2$, $\beta_1=\beta_2$;
 \item $n\geq3$,  $\chi(S, \beta)< \min\{2, 2+2\min_{i} \beta_i\}$.
     \end{itemize}
Note that surfaces in the first class are often called (American) footballs, and they belong to the critical case. In~\cite{CL2}, it has been shown a conic metric of positive constant curvature on $S^2$ in the critical case is necessarily a football. For the second class, the sufficiency  is proved by Troyanov~\cite{Tr}, the necessity and uniqueness argument is due to Luo-Tian\cite{LT}. 

Recently, there are renewed interests on metrics with conic singularities in the study of K\"{a}hler geometry, namely the K\"{a}hler-Einstein metrics with cone singularity along a divisor. Conic metrics of constant curvature on Riemann surfaces are just one-dimensional examples of K\"{a}hler-Einstein metrics with cone singularity along a divisor. In higher dimension, for the existence of K\"{a}hler-Einstein metrics with cone singularities, Troyanov's condition can be generalized to the coercivity of twisted Mabuchi $K$-energy functional~\cite{JMR}, which can be reinterpreted as the pair $(S,D)$ being logarithmically K-stable~\cite{RT}. The smooth version of this connection between algebraic stability and existence of K\"{a}hler-Einstein metrics on Fano manifolds is the essence of the recently solved Yau-Tian-Donaldson conjecture~\cite{CDS1, CDS2, CDS3, T2}.



From this point of view, Troyanov's classification can be understood as the stability conditions. The subcritical pairs and critical pairs can be viewed as  being stable and semi-stable respectively; while the supercritical pairs are considered being unstable. While there exist metrics with constant curvature on stable and semi-stable pairs with two conic points; the canonical metric problem for unstable pairs have less definite answers. One candidate of the canonical metrics on unstable pairs  can be conical gradient shrinking Ricci solitons with two conic points. This is verified in the recent preprints ~\cite{PSSW1,PSSW2} via the Ricci flow. See ~\cite{Y1, Y2, MRS} for previous works on the Ricci flow on conic surfaces.



In this note, we investigate from another perspective. We consider the moduli space of all conic spheres with constant curvature $1$, which we call $\mathcal M$. According to numbers of the conic points on the sphere, $\mathcal M$ has complicated topology with  components with varying dimensions. Algebraically, $\mathcal M$ can be separated as stable and semi-stable parts. Fix any smooth metric $g'$ on the 2-sphere. For a sequence of divisors $D_{l}=\sum_{i=1}^{n}\beta_{i,l}p_{i,l}$, $l=1,2,\cdots,$ we write
$$\lim_{l\to\infty }D_{l}=D=\sum_{i=1}^{n}\beta_{i}p_{i},$$
if for each $1\leq i\leq n$, we have $\lim_{l\to\infty }\beta_{i,l}=\beta_{i}$ and $\lim_{l\to\infty }p_{i,l}=p_{i}$ with respect to $g'$. Obviously this definition is independent of choice of $g'$.

We show that any sequence of spheres of positive constant curvature with $n$ ($n\geq3$) conic points passing from stable case to semi-stable case converge to a football in Gromov-Hausdorff topology. Geometrically, all but one conic points will merge into a single conic point of the limit football. The same phenomenon occurs for the Ricci flow on semi-stable conic spheres~\cite{PSSW1}. Following the method of Luo-Tian~\cite{LT}, we have the following
\begin{thm} \label{t1} Let $(S^2, g_l, D_{l}=\sum_{i=1}^{n}\beta_{i,l}p_{i,l})$ be a sequence of Riemann spheres with conic metric of constant curvature $K=1$. Suppose $D_{l}\to D=\sum_{i=1}^{n}\beta_{i}p_{i}$, which is a divisor belonging to the critical case. Suppose $\beta_{1}=\min_{i} \beta_{i}$.
Then a subsequence of  $(S^2, g_{l})$ converges in the Gromov-Hausdorff topology to  $(S^2, g, D')$, the unique football of constant curvature $K=1$ with $D'=\beta_{1} p+ \beta_1 q$. Moreover, suppose $\beta_{1,l}\leq \cdots\leq \beta_{n,l}$, then the corresponding conic points converge in the following fashion:
\[
\lim_{l\to \infty} p_{1,l}=p \quad \text{and} \quad \lim_{l\to \infty} p_{i,l}=q, \quad \text{for $i\geq2$}.
\]
\end{thm}

To understand this convergence phenomenon in an analytical way, we investigate the problem in the conformal geometrical setting.

Let $g_0$ be the standard Euclidean metric. Under  the stereographic projection, a conformal  metric $g=e^{2u}g_0$ is of constant curvature $1$  and represents $(S^2, g,D=\sum_{i} \beta_i p_i$) if and only if $u$ satisfies the equation
\begin{align} \label{e2.1}
\Delta u=-e^{2u} , \quad z\in \mathbb{C}\setminus\{z_1, \cdots, z_n\},
\end{align}
with the  asymptotic behavior of $u$ near $z_i$ being:
\begin{itemize}
\item $u \sim \beta_i \ln |z-z_i|$ as $z\to z_i$;
\item $u \sim -2\ln|z|$ as $|z|\to \infty$.
\end{itemize}

Notice that $u$ is uniquely associated to a conic metric up to conformal transformations. Our main result is the following

\begin{thm} \label{main}Let $u_{l}$ be functions on $\mathbb {C^{*}}$ representing $g_{l}$ given in Theorem~\ref{t1}.
Under proper normalization, there exist two distinct points $p, q\in\mathbb C^{*}$ such that a subsequence of $u_{l}$ converges to $u_{\infty}$ on any compact set $K \subset  \mathbb {C^{*}}\setminus \{p,q\}$ with $e^{2u_{\infty}}g_{0}$ on $\mathbb {C^{*}}-\{p,q\}$ representing a football.
\end{thm}

The standard stereographic projection gives a unique correspondence of metrics on sphere and a conformal factor function on ${\mathbb C}^{*}$ up to a M\"obius transformation. This large conformal transformation group poses analytical difficulty, namely a proper conformal gauge must be chosen properly so that the resulting limit for the conformal factors exists and is non-trivial. 


Our main approach is to explore the rotational symmetry of the football solution. This was examined in~\cite{CL2} when a single manifold with conic singularities is considered. Especially, level sets of the conformal factor are considered and isoperimetric inequality plays a crucial role. For our set-up, we follow a similar but more delicate approach. We would like to consider the level sets of each conformal factor $u_{l}$ and analyze the isoperimetric defect more carefully. Consequently, our choice of normalization is also connected to the level sets. See Sect.3 for more details.


While the method applied in this article heavily relies on the rotational symmetry of football solutions, it is a perfect example highlighting the equivalence of different convergence concepts on the moduli space of constant curvature metrics with conic singularities. Several set-up for higher dimensions can be considered from both K{\"a}hler geometry and conformal geometry points of views.


This paper is organized as follows: in Section 2, we provide a proof of Theorem~\ref{t1} following~\cite{LT} and in Section 3 we prove Theorem~\ref{main}, which also leads to an alternative proof of Theorem~\ref{t1}.

{Acknowledgements: Both authors would like to thank Jian Song and Lihe Wang for discussion. Both authors thank the referee for useful comments. Part of the work was done when both authors were visiting Beijing International Center for Mathematical Research. We are thankful for its  hospitality. }
\section{Proof of Theorem~\ref{t1}}
We adopt the geometric setting of~\cite{LT} and notations therein. By a theorem of Alexandrov, each spherical conic metric of constant curvature $1$ is isometric to the boundary of a convex polytope in $S^3$. There are two degenerate cases: one is the metric doubling of a "lens", which is a degenerate spherical triangle with length of three sides being $\pi,\pi,0$; the other one is the metric doubling of a usual spherical triangle. Clearly the former one corresponds to a football (2-conic points) and the latter one corresponds to a sphere with three conic points.

For a convex polytope of $n$ vertices, we denote its angles at vertices by $(\alpha_1, \cdots \alpha_n)$ with each $\alpha_i\in (0,2\pi)$. Let $\mathcal{P}_n$ be the space of all boundaries of labeled $n$-vertex convex polytopes in $S^3$ modulo isometry, with the topology  induced by the Hausdorff metric. For each convex polytope $P$, construct a totally geodesic triangulation, 
then there are exactly $3(n-2)$ edges and $2(n-2)$ triangles. Variation of the length of each edge gives rise to distinct convex polytopes (up to isometry). Therefore, the dimension of $\mathcal{P}_n$ is $3(n-2)$. Meanwhile, denote the conformal structure of $n$-labeled Riemannian sphere by $\mathcal{M}_n$. Since M\"obius transformations are $3$-transitive, it follows that $\dim \mathcal{M}_n=2(n-3)$.

In~\cite{LT}, Luo-Tian  show that there exists a conic metric on $S^2$ of positive constant curvature representing $D=\displaystyle \sum_{i=1}^k\beta_{i}p_i$, $(k\geq3)$, if and only if the corresponding cone angles $\alpha_i=2\pi(1+\beta_i)$ satisfy
\begin{align} \label{condition}
\sum_{i=1}^{n} \alpha_i>2(n-2)\pi, \quad \sum_{i=1}^{n} \alpha_i<2(n-2)\pi+2\min_{i}\alpha_i.
\end{align}

This condition is exactly the same as the subcritical condition of Troyanov. It defines a convex open set in $\mathbb{R}^n$, which we denote by $\mathcal{A}$. The critical case corresponds to the equality
\[
\sum_{i=1}^{n} \alpha_i=2\pi(n-2)+2\min_{i}\alpha_i.
\]

In addition, Luo-Tian~\cite{LT} have proved the following

\begin{thm}[Luo-Tian]\label{thm2.1}
The map
\begin{align}\label{map}
\Pi: \mathcal{P}_n &\to \mathcal{M}_n\times \mathcal{A},\\
\notag   P &\to ( \text{conformal structure of $P$} , \text{angles of $P$ at vertices})
\end{align}
is a homeomorphism.
\end{thm}

We are now ready to give a proof of Theorem~\ref{t1}.

\begin{proof}[Proof of Theorem ~\ref{t1}] As mentioned above, each $(S^2, g_l, D_{l})$ is isometric to the boundary of a convex polytope $P_{l}$ in $S^3$, and satisfies the subcritical condition (\ref{condition}). By compactness of compact sets in $S^3$ with respect to the Hausdorff metric, we may assume a subsequence, still denoted by $P_{l}$, converges to a convex polytope $P_{\infty}$. $P_{\infty}$ represents a conic sphere of positive constant curvature, which is either in the subcritical case or the critical case. For the latter, $P_{\infty}$ must be a 'lens'.
By Theorem~\ref{thm2.1}, if $P_{\infty}$ does not degenerate to a lens, then
\[
D_{\infty}=\lim_{l\to \infty} D_{l}=\lim_{l \to \infty} \Pi_2(P_{l})=\Pi_2(P_{\infty})
\] must be in the subcritical case as well, a contradiction. Thus, we conclude that $P_{\infty}$ is a lens.

Thus to prove our result, we are left to show that the conic angle of the corresponding football is $2\pi(\min_{i}\beta_{i,\infty}+1)$. Let $\alpha_{i,l}=2\pi(\beta_{i,l}+1)$ be the angles of $P_{l}$, with corresponding vertices denoted by $V_{i,l}$, $i=1,2,\cdots,n$. It was shown~\cite{LT} under this situation that there exists $k$ such that
\begin{align} \label{k1}
\lim_{l\to \infty} d(V_{k,l}, V_{i,l})>0, \quad \text{for $i\neq k$},
\end{align} and
\begin{align} \label{k2}
\lim_{l\to \infty} d(V_{i,l}, V_{j,l})=0, \quad \text{for $i,j \neq k$}.
\end{align}

Now consider the triangulation of $P_{l}$ which consists of $2(n-2)$ triangles $ \{\Delta_{s,l}\}_{s=1}^{2(n-2)}$.  For each $\Delta_{s,l}$ we denote its inner angles as $\alpha^{1}_{s,l},\alpha^{2}_{s,l},\alpha^{3}_{s,l}$. Let $V_{k,l}$ be the vertex as given by (\ref{k1}) and (\ref{k2}). For triangles $\Delta_{s,l}$ which are incident to $V_{k,l}$, we shall let $\alpha^{1}_{s,l}$'s be the angles at the point $V_{k,l}$. For $k=1,2,3$ and $i=1,2,\cdots, 2(n-2)$, let
$$\alpha^{k}_{s,\infty}=\lim_{l\to\infty} \alpha^{k}_{s,l}.$$
When $l$ is converging to $\infty$, we have two different cases depending on whether a triangle is incident to the $V_{k,l}$ or not.

If $\Delta_{s,l}$  is incident to $V_{k,l}$, then one side of $\Delta_{s,l}$ collapses so $\Delta_{s,l}$ converges to a lens. Moreover the corresponding limit angles satisfy
\begin{align} \label{3.4}
\alpha^{1}_{s,\infty}=\alpha^{2}_{s,\infty}+\alpha^{3}_{s,\infty}-\pi.
\end{align}
If $\Delta_{s,l}$ is not incident to $V_{k,l}$, then all its three vertices merge in the limit. Hence limits  of its inner angles satisfy
\begin{align} \label{3.5}
\alpha^{1}_{j,\infty}+\alpha^{2}_{j,\infty}+\alpha^{3}_{j,\infty}=\pi.
\end{align}
Summing relations (\ref{3.4}) and (\ref{3.5}) for all $2(n-2)$ triangles,  we have
\[
\alpha_{k,\infty}+2(n-2)\pi=\sum_{i\neq k} \alpha_{i,\infty}.
\]
Comparing with the critical condition, it follows that
\[
\alpha_{k,\infty}=\min_{1\leq i\leq n} \alpha_{i,\infty}.
\]
We have thus finished the proof.
\end{proof}

\section{Proof of Theorem~\ref{main}}
In this section, we give a proof of Theorem~\ref{main}. Under the stereographic projection, the conformal factors of  positive constant curvature metrics are solutions to a semi-linear elliptic equation in the complex plane $\mathbb{C}$. We shall study the limit behavior of these solutions.

More precisely, given a conic sphere $(S^2, g, D)$, under the stereographic projection we assume $z_i$'s are the corresponding projection of $p_i$ in the complex plane. Let $g_0$ be the standard Euclidean metric, then $\frac{4}{(1+|z|^2)^2}g_0$ is the standard metric on $S^2$. A conic metric $g=e^{2u}g_0$ is of constant curvature $K=1$ representing $D=\sum_{i=1}^n \beta_i p_i$ if and only if $u$ satisfies the equation
\begin{align} \label{e2.1}
\Delta u=-e^{2u} , \quad z\in \mathbb{C}\setminus\{z_1, \cdots, z_n\},
\end{align}
with the asymptotic behavior of $u$ near $z_i$ being:
\begin{itemize}
\item $u \sim \beta_i \ln |z-z_i|$ as $z\to z_i$;
\item $u \sim -2\ln|z|$ as $|z|\to \infty$.
\end{itemize}

Note that a conformal factor $u_{\alpha}$ for the football of cone angle $\frac{2\pi}{\alpha}$ and $K=1$ can be easily written out by pulling back $\frac{4}{(1+|z|^2)^2}g_0$ using the map $z\to z^{\alpha}$, namely
 $e^{2u_{\alpha}}g_{0}$ represents a football of cone angle $\frac{2\pi}{\alpha}$ and $K=1$, where
\begin{equation}
e^{2u_{\alpha}}= 4\alpha^2 \frac{|z|^{2\alpha-2}}{(1+|z|^{2\alpha})^2}.\label{football}
\end{equation}

 For a sequence of conic metrics of constant curvature $1$ on $S^2$ representing $D_{l}=\sum_{i=1}^{n}\beta_{i,l}p_{{i,l}}$ and $\lim_{l\to \infty} D_{l}=D_{\infty}$, we may assume, without loss of generality, that $\min_{i} \beta_{i,l}=\beta_{1,l}$. We also define $\alpha_{l}=|D_{l}|-\beta_{i,l}=\sum_{i=2}^{n}\beta_{i,l}$ and $\alpha_{\infty}= |D_{\infty}|-\beta_{1,\infty}$. Therefore $D_{l}$ being subcritical  is equivalent to $\beta_{1,l}>\alpha_{l}$; $D_{\infty}$ being critical is equivalent to  $\beta_{1,\infty}=\alpha_{\infty}$.

For each $l$, by the conformal description above,  and assuming that we fix $z_{1,l}$ at $\infty$,
we have $g_l=e^{2u_l} g_0$ where $u_l$ is the solution of
\begin{align} \label{e2.3}
\Delta u_l=-e^{2u_l} \quad \text{in $\mathbb{C}\setminus \{z_{2,l}, \cdots, z_{n,l}\}$ },
\end{align}
subject to the asymptotic behavior
\begin{itemize}\label{singularity}
  \item $u_l\sim \beta_{i,l}\ln|z-z_i^{(l)}|$ as $z\to z_{i,l}$ for $i=2, \cdots, n$;
  \item $u_l\sim -(2+\beta_{1,l})\ln|z|$ as $|z|\to \infty$.
\end{itemize}

The difficulty of the conformal geometry on sphere lies in the fact that the group of conformal transformations is non-compact. Hence generically, the sequence of $u_l$ does not have any convergent subsequence. Also $u_l$ in (\ref{e2.3}) is not unique. In particular, for
\begin{align}\tag{$\dagger$}
  \text{scaling} \quad & u^{\lambda,0}(z):= u(\lambda z)+\ln \lambda; \\ \tag{$\ddagger$}
  \text{translation} \quad &u^{0,\kappa}(z):=u(z-\kappa),
\end{align}
$e^{2u^{\lambda}}g_{0}$ and $e^{2u^{\kappa}}g_{0}$ all represent the same conic metric on the punctured sphere as  $e^{2u}g_{0}$.


To clearly state the normalization we shall choose, we first present the main tools of the proof: to study the level sets of $u_l$ and apply the isoperimetric inequality. While these ideas have been explored before ~\cite{CL2}, our problem requires more delicate analysis. We also refer the reader to ~\cite{BL} where similar ideas and tools are used to treat the mean field equation. We would examine the defect of isoperimetric inequality carefully under the limit procedure. With the help of the proper normalizations,  we prove the Hausdorff convergence of level sets. This convergence leads to a uniform bound of $u_{l}$ on compact sets, which allows us to extract a limit function $u_{\infty}$. Then the isoperimetric inequality is applied again to prove that $u_{\infty}$ must be radially symmetric about some point $z_0$.

For each $u_l$, let $\Omega^{u_l}(t):=\{ u_l>t\}$,
\[
A^{u_{l}}(t):=\int_{\Omega_l(t)} e^{2u_l} \quad \text{and} \quad B^{u_l}(t)=\int_{\Omega^{u_{l}}(t)} 1=|\Omega_{l}(t)|.
\]
Thus $A^{u_{l}}$ is monotone decreasing and the Gauss-Bonnet formula yields
\begin{align}
\int_{\mathbb{R}^2} e^{2u_l} = 2\pi(2+|D_{l}|).
\end{align}
It follows that
\[
\lim_{t\to -\infty} A^{u_{l}}(t)=2\pi(2+|D_{l}|) \quad \text{and}\quad \lim_{t\to \infty} A^{u_l}(t)=0.
\]

Under the scaling and translation, we have
$$
A^{u_{l}^{\lambda,\kappa}}(t)=A^{u_l}_l(t-\ln \lambda).
$$

We can now state our normalization for all $u_{l}$. Pick a generic real number $t^{*}$, and for each $l$ choose suitable $\lambda_l$ and $\kappa_{l}$ such that

\begin{align} \label{n1}
A^{u^{\lambda_{l},\kappa_{l}}}(\ln(1+\beta_{1,\infty}))=\pi(2+|D_{l}|)=\frac{1}{2}A^{u^{\lambda_{l},\kappa_{l}}}(-\infty).
\end{align}

\begin{equation}
{\text {The centroid of }}\Omega^{u^{\lambda_{l},\kappa_{l}}}(t^{*}) \text{ is at } 0.
\end{equation}

Here the centroid of a region $\Omega\subset\mathbb C$ is  the point$(\int_{z\in \Omega}z )/(\int_{z\in \Omega} 1)\in\mathbb C$. 

From now on, without confusion we write $u_{l}$ for $u^{\lambda_{l},\kappa_{l}}$ and we write
 $$\Omega_{l}(t)=\Omega^{u_{l}}(t),\ \ A_{l}(t)=A^{u_{l}}(t),\ \ B_{l}(t)=B^{u_{l}}(t).$$

 We can thus restate Theorem~\ref{main} as the following:
\begin{thm} \label{t2} For a sequence of functions $\{u_{l}\}$ satisfying (\ref{e2.3}), assume that
\begin{align} \label{newn1}
A_{l}(\ln(1+\beta_{1,\infty})) =\pi(2+|D_{l}|)=\frac{1}{2}A_{l}(-\infty),
\end{align}

\begin{equation}\label{norm2}
{\text {The centroid of }}\Omega_{l}(t^{*}) \text{ is at } 0,
\end{equation}
where $t^{*}\in\mathbb R$ is a fixed generic point, then $u_l$ sub-converges to $u_{\infty}$ in $C^{\infty}_{loc}(\mathbb{C}\setminus \{0\})$, where $u_{\infty}$ is given by
\begin{align} \label{3.2}
e^{2u_{\infty}(z)}=4(1+\beta_{1,\infty})^2 \frac{|z|^{2\beta_{1,\infty}}}{(1+|z|^{2+2\beta_{1,\infty}})^2}.
\end{align} Moreover,
\[
\lim_{l\to \infty} z_{i,l}=0, \quad \text{for $i\geq2$}.
\]
\end{thm}

In view of (\ref{football}), Theorem~\ref{t1} is an immediate consequence of Theorem~\ref{t2}.

In the rest of this section, we present a proof of Theorem~\ref{t2}.

we take a careful look at $\Omega_l(t)$. In view of the asymptotic behavior of $u_l$, $\Omega_l(t)$ is a bounded region for each $t$, and
\begin{align} \label{2.2}
A_l(t)=\int_{\Omega_l(t)} e^{2u_l} =\int_{\Omega_l(t)} -\Delta u_l=\int_{\partial \Omega_l(t)} |\nabla u_l|+2\pi\alpha_{l}.
\end{align}

In general, there are multiple connected components for $\Omega_l(t)$. For a regular value $t$ of $u_l$, $\Omega_l(t)$ consists of finitely many disjoint regions bounded by Jordan curves. Each component is simply connected due to the maximum principle.

We now present the following estimate relating the size of level sets and  the upper bound of the function.

\begin{lem} \label{l3.3} For a fixed $t_0\in \mathbb R$, let $\Gamma_l(t_0)$ be a connected component of $\Omega_l(t_0)$ does not contain any singular point of $u_l$ and $H_{l}=\max_{\Gamma_l(t_0)} \{u_l\}$.  For $t\in[t_0,H_l]$, let $\Gamma_l(t)=\Omega_l(t)\cap \Gamma_l(t_0)$, $a_l(t)=\int_{\Gamma_l(t)}e^{2u_l}$ and $b_l(t)=|\Gamma_l(t)|$, then we have

\begin{equation}\label{3a}
a_l(t)\geq 4\pi(1-e^{t-H_l}) .
\end{equation}
Furthermore, for $a_l(t)\leq 2\pi$, we have
\begin{equation}
b_l(t)\geq 4\pi e^{-H_l}(e^{-t}-e^{-H_l})\label{3b}.
\end{equation}
\end{lem}

\begin{proof}
Since $\Gamma_l(t_0)$ does not contain any singularity, a similar computation of (\ref{2.2}) shows
\begin{align} \label{2.3}
a_l(t)=\int_{\Gamma_l(t)} e^{2u_l} =\int_{\Gamma_l(t)} -\Delta u_{l}=\int_{\partial \Gamma_l(t)} |\nabla u_l|, \quad t\in[t_0, H_l],
\end{align}
and $a_l(H_l)=0$.

By the co-area formula, we have
\begin{align}
a_l'(t)=-e^{2t}\int_{\partial \Gamma_l(t)} \frac{1}{|\nabla u|},
\end{align}
and 
\begin{align}
b_l'(t)=-\int_{\partial \Gamma_l(t)} \frac{1}{|\nabla u|}.
\end{align}
Hence
\begin{align} \label{2.5}
(a_l(t)^2)'&= 2a_l(t) a_l'(t)=-2e^{2t} \int_{\partial \Gamma_l(t)} |\nabla u_l| \int_{\partial \Gamma_l(t)} \frac{1}{|\nabla u_l|}\\ \notag
&\leq -2e^{2t} (\int_{\partial \Gamma_l(t)} 1)^2\leq -8\pi e^{2t} |\Gamma_l(t)|=-8\pi e^{2t}b_l(t).
\end{align}
Here we have used H\"older's inequality and the isoperimetric inequality for $\Gamma_l(t)$:
\begin{align} \label{e3.5}
\int_{\partial \Gamma_l(t)}|\nabla u_l| \int_{\partial \Gamma_l(t)} \frac{1}{|\nabla u_l|}\geq |\partial \Gamma_l(t)|^2\geq 4\pi |\Gamma_l(t)|.
\end{align}

By Fubini's theorem, we also have
\begin{align} \label{2.3}
a_{l}(t)=\int_{\Gamma_l(t)} e^{2u_l}&=\int_{0}^{\infty} |e^{2u_l}>\lambda| d\lambda  \\ \notag
&=\int_{-\infty}^{H} |u> t|2e^{2t} dt\\ \notag
 &=e^{2t}b_l(t)+\int_{t}^{H}2e^{2t} b_l(t) dt.
 \end{align}
Integrating  (\ref{2.5}) from $t$ to $H_l$ and using (\ref{2.3}), we obtain
\begin{equation}\label{2.33}
-a_l(t)^2\leq -4\pi a_l(t)+4\pi e^{2t}b_l(t).
\end{equation}

Combining (\ref{2.5}) and (\ref{2.33}), we have
\begin{align}\label{diffineq}
-a_l(t)\leq -4\pi -a_l'(t).
\end{align}
(\ref{3a}) then follows from (\ref{diffineq}) and the fact that  $a_{l}(H)=0$.
When $a_{l}\leq 2\pi$, (\ref{3b}) is thus a consequence of (\ref{3a}) and (\ref{2.33}).
\end{proof}

\begin{rem} Lemma~\ref{l3.3} is used later to show the uniform upper bound for $u_l$ on any compact set $K$ which does not contain any singularities of $u_l$. The general form of such estimate has been obtained by Brezis-Merle~\cite{BM}. 
\end{rem}
\begin{lem} \label{l3.1}
Let
\begin{align}
A_{\infty}(t):=4\pi(1+\beta_{1,\infty})\frac{\rho^{2+2\beta_{1,\infty}}}{1+\rho^{2+2\beta_{1,\infty}}},
\end{align}
where $\rho$ is determined by $e^{2t}=4(1+\beta_{1,\infty})^2\frac{\rho^{2\beta_{1,\infty}}}{(1+\rho^{2+2\beta_{1,\infty}})^2}$, then under the normalization (\ref{n1}), we have
\[
\lim_{l\to \infty} A_l(t)=A_{\infty}(t), \quad \text{$\forall t\in \mathbb{R}$},
\]
and
\[
\lim_{l \to \infty} B_l(t)=B_{\infty}(t)=\pi\rho^2.
\]
\end{lem}

\begin{proof} We now run a similar argument for level sets including singular points.

Let
 \begin{align} \label{e3.6}
 f_l(t):=A_l^2(t)-(4\pi+4\pi\alpha_l) A_l(t)+4\pi e^{2t} B_l(t).
 \end{align}
Then
\begin{align}  \label{e3.2} \notag
f_l'(t) =&-2e^{2t}\int_{\Omega_l(t)} e^{2u_l}\int_{\partial \Omega_l(t)} \frac{1}{|\nabla u_l|} + (4\pi+4\pi\alpha_l) e^{2t}\int_{\partial \Omega_l(t)}\frac{1}{|\nabla u_l|} \\ \notag
&+8\pi e^{2t}B_l-4\pi e^{2t} \int_{\partial \Omega_l(t)} \frac{1}{|\nabla u_l|} \\ \notag
=&-2e^{2t}(\int_{\partial \Omega_l(t)} |\nabla u_l|+2\pi\alpha_l)\int_{\partial \Omega_l(t)} \frac{1}{|\nabla u_l|} \\ \notag
&+(4\pi+4\pi\alpha_l) e^{2t}\int_{\partial \Omega_l(t)}\frac{1}{|\nabla u_l|}+8\pi e^{2t}B_l-4\pi e^{2t} \int_{\partial \Omega_l(t)} \frac{1}{|\nabla u_l|} \\
=&2e^{2t}(4\pi B_t- \int_{\partial \Omega_l(t)}|\nabla u_l| \int_{\partial \Omega_l(t)} \frac{1}{|\nabla u_l|})\leq 0.
\end{align}

Since
\[
\int_{\mathbb{R}^2} e^{2u_l} dx=\int_{-\infty}^{\infty}2e^{2t}B_l(t) dt <\infty,
\] it follows that
\[
e^{2t}B_l(t)\to 0, \quad \text{as $t \to \pm \infty$}.
\]
Let $C_l=2\pi(2+|D_{l}|)=\lim_{t\to -\infty} A_l(t)$, we get
\[
0=\lim_{t\to\infty}f_{l}(t)\leq f_l(t)\leq  \lim_{t\to-\infty}f_{l}(t)=C_l^2-(4\pi+4\pi\alpha_l)C_l.
\]
The convergence $D_{l}\to D$ implies
\[
\lim_{l \to \infty}C_l^2-(4\pi+4\pi\alpha_l)C_l=\lim_{l \to \infty}2\pi C_l(\beta_{1,l}-\alpha_{l})=0.
\]
Therefore, $f_l$ converges to $0$ uniformly.  Moreover $f'_l$ is integrable with
\[
 \lim_{l\to \infty} ||f'_l||_{L^1}=0.
\]

Combining (\ref{e3.6}) and (\ref{e3.2}) we find that $A_l$ satisfies
\begin{align} \label{e3.3}
A_lA_l'-(2\pi\alpha_l)A_l'-(A_l^2-(4\pi+4\pi\alpha_l)A_l)=\frac{1}{2}(f'_l-2f_l).
\end{align}
(\ref{e3.3}) can be re-arranged as
\begin{align} \label{e10}
(\frac{a}{A_l}+\frac{b}{4\pi(1+\alpha_l)-A_l})A_l'=1+\frac{\frac{1}{2}f'_l-f_l}{A_l^2-4\pi(1+\alpha_l)A_l},
\end{align}
where $a=\frac{\alpha_l}{2\alpha_l+2}$ and $b=-\frac{\alpha_l+2}{2\alpha_l+2}$.

Let
\[
 t_l=\sup_t\{t|\liminf_l 4\pi(1+\alpha_l)-A_l(t)=0\} \quad \text{and} \quad t_r=\inf_t\{t|\liminf_l A_l(t)=0\}.
\]
Due to (\ref{newn1}), we  have $t_l<\ln (1+\beta_{1,\infty})<t_r$.
It follows from the definition of $t_l$ and $t_r$ that  $ \frac{1}{A_l^2-4\pi(1+\alpha_l)A_l}$ is uniformly bounded on any finite closed interval  $[r,s] \subset (t_l, t_r)$.

Recall that $\displaystyle \lim_{l \to \infty}||f'_l||_{L^1}=0$,  $f'_l\to 0$ in $L^1$ and we also have $f_l$ converges uniformly to $0$. Hence integrating (\ref{e10}) and taking the limit, we obtain
\begin{align} \label{e3.4}
\lim_{l\to \infty} \ln (A_l^a(4\pi(1+\alpha_l)-A_l)^{-b})|_{r}^s= \lim_{l\to \infty}\int_{r}^s 1+\frac{\frac{1}{2}f'_l-f_l}{A_l^2-4\pi(1+\alpha_l)A_l} dt=s-r.
\end{align}
Notice that $\displaystyle \lim_{l\to \infty}A_l(\ln(1+\beta_{1,\infty}))=2\pi(1+\beta_{1,\infty})$, from this single point convergence and (\ref{e3.4}) we conclude that $A_l$ has a pointwise limit $A_{\infty}$ on $(t_l, t_r)$ which satisfies
\begin{align} \label{3.6}
(\frac{a}{A_{\infty}}+\frac{b}{4\pi(1+\alpha_l)-A_{\infty}})A_{\infty}'=1
\end{align}
with $A_{\infty}(\ln(1+\beta_{1,\infty}))=2\pi(1+\beta_{1,\infty})$.

By separation of variables, we find the solution of (\ref{3.6})
\begin{align}
A_{\infty}(t)=4\pi(1+\beta_{1,\infty})\frac{\rho^{2+2\beta_{1,\infty}}}{1+\rho^{2+2\beta_{1,\infty}}},
\end{align}
where $\rho$ is  chosen such that $e^{2t}=4(1+\beta_{1,\infty})^2\frac{\rho^{2\beta_{1,\infty}}}{(1+\rho^{2+2\beta_{1,\infty}})^2}$. It is easy to see that
\begin{align} \label{2.6}
\lim_{t\to -\infty} A_{\infty}(t)=4\pi(1+\beta_{1,\infty}) \quad \text{and} \quad \lim_{t\to \infty} A_{\infty}(t)=0.
\end{align}

Combining (\ref{e10}) and (\ref{2.6}), it is obvious that  $(t_l,t_r)=(-\infty,\infty)$. A simple computation which we shall omit here gives the corresponding result  for $B_{l}(t)$.

\end{proof}

We now study the isoperimetric defect. Define, for any region $\Omega\subset{\mathbb R}^{2}$ with boundary a Jordan curve $\partial\Omega$, the isoperimetric defect is
\begin{equation}
D(\Omega):=|\partial \Omega|^2-4\pi|\Omega|.
\end{equation}
It is easy to show that $D(\Omega)$ is super additive. This means, if $\Omega_{1}$ and $\Omega_{2}$ are two disjoint sets in ${\mathbb R}^{2}$,  $\Omega=\Omega_{1}\cup\Omega_{2}$, we have

\begin{equation} D(\Omega)\geq D(\Omega_{1})+D(\Omega_{2}).\label{super+}
\end{equation}

Furthermore, we have the following
\begin{lem}[Bonnesen's inequality]
 For a bounded region $\Omega\subset{\mathbb R}^{2}$, let $r$ and $R$ be the radii of incircle and circumcircle of $\Omega$, then
\[
D(\Omega)\geq \pi^2(R-r)^2;
\]
The equality holds if and only if $\Omega$ is a round disk.
\end{lem}

In view of Bonnesen's inequality, we can prove
\begin{lem}\label{l3.2} Let $D_l(t):=D(\Omega_{l}(t))$ be the isoperimetric defect of the level set $\Omega_{l}(t)$. Then there exists a subset $V\subset \mathbb{R}$ 
such that  $|\mathbb{R}\setminus V|=0$ and after passing to a subsequence,
\begin{align} \label{2.4}
\lim_{l\to \infty} D_l(t)=0, \quad \forall t\in V.
\end{align}
\end{lem}

\begin{proof} By  (\ref{e3.2}) and (\ref{e3.5}), we have
\begin{align}
2e^{2t} D_l(t) \leq 2 e^{2t}(\int_{\partial \Omega_l(t)} |\nabla u_l| \int_{\partial \Omega_l(t)} \frac{1}{|\nabla u_l|}-4\pi|\Omega_{l}(t)|) \leq -f_l'(t).
\end{align}

For each fixed $t_0$, we then conclude
\[
2e^{2t_0}\int_{t_0}^{\infty} D_l(t) dt\leq \int_{t_0}^{\infty} 2e^{2t}B_l(t) dt\leq \int_{t_0}^{\infty}-f_l'(t) dt =f_l(t_0).
\]
Since $f_l(t_0)\to 0$ as $l \to \infty$, we have $D_l(t)$ converges to $0$ in $L^{1}_{(t_{0},\infty)}$ norm, thus after passing to a subsequence $D_l(t)$ converges to $0$ almost everywhere for $t\geq t_{0}$. Repeating the same argument for a sequence of $t_{i}\to-\infty$ and using a diagonal argument, it follows $D_l(t)$ converges to $0$ almost everywhere on $\mathbb{R}$. Thanks to Sard's theorem, after disregarding critical values for all of $u_l$, we still get the convergence (\ref{2.4}) almost everywhere.
\end{proof}

\begin{lem}\label{l3.4} For each $t\in V$, where $V$ is given as in Lemma~\ref{l3.2},  let $\Sigma_{l}(t)$ be the connected component of  $\Omega_{l}(t)$ with largest area. Then
$$|\Sigma_{l}(t)|\to B_{\infty}(t), \quad  \text{as} \quad l\to\infty,$$
 $$|\Omega_{l}(t) \setminus \Sigma_{l}(t)|\to 0\quad  \text{and} \quad |\partial(\Omega_l(t) \setminus \Sigma_l(t))|\to 0, \quad \text{as} \quad  l\to\infty.$$
\end{lem}

\begin{proof} We prove by contradiction. If for $\Omega^{1}_l(t)=\Sigma_{l}(t)$ and $\Omega^{2}_{l}(t)=\Omega_l(t)\setminus\Sigma_{l}(t)$  we have
\[
\liminf_{l \to \infty} |\Omega^1_l(t)|\geq \delta_1>0\quad \text{and} \quad \liminf_{l \to \infty} |\Omega^2_l(t)|\geq \delta_2>0.
\]
Then
\begin{align}
D_l(t)&= (|\partial \Omega^1_l(t)|+|\partial \Omega^2_l(t)|)^2-4\pi(|\Omega^1_l(t)|+|\Omega^2_l(t)|) \\ \notag
 &=( |\partial \Omega^1_l(t)|^2-4\pi|\Omega^1_l(t)|)+(|\partial \Omega^2_l(t)|^2-4\pi|\Omega^2_l(t)|)+2|\partial \Omega^1_l(t)||\partial \Omega^2_l(t)| \\ \notag
 &\geq 2\sqrt{4\pi |\Omega^1_l(t)|} \sqrt{4\pi|\Omega^2_l(t)|}\geq 8\pi \sqrt{\delta_1\delta_2}>0,
\end{align}
a contradiction to Lemma~\ref{l3.2}.

Hence we conclude that there is exactly one component whose area tends to $B_{\infty}(t)$, which we denote by $\Sigma_{l}(t)$.  Moreover, both the area and the boundary length of the remaining components must go to zero.\end{proof}


Now take a monotone sequence $\{t_i\}_{i\in\mathbb{Z}}\subset V$, where $V$ is obtained in Lemma~\ref{l3.2}. By Lemma~\ref{l3.1} and Lemma~\ref{l3.4}, for any $0<\lambda<<1/2$, there exists a positive integer $L_i$ such that for all $l>L_{i}$,
\begin{equation}\frac{|\Sigma_l(t)|}{|\Omega_l(t)|}\geq 1-\lambda.\label{unique}
\end{equation}
Such a $\Sigma_{l}(t)$ is thus unique. By a diagonal  argument, we may pick a subsequence of $u_{l}$ (which we still call $u_{l}$) and assume that (\ref{unique}) holds for all $l$. Notice that
for $t_{i}>t_{j}$, we have
\begin{equation}\Sigma_{l}(t_{i})\subset\Sigma_{l}(t_{j}).
\label{add1}
\end{equation}

We now explain our choice of $t^{*}$ in (\ref{norm2}). Without loss of generality, let $t^{*}=t_{0}\in V$. We have thus the following

 \begin{lem}  \label{l3.5}  There exists a sequence of descending balls
\[
\cdots \supset B_{r_{i-1}}(p_{i-1})\supset B_{r_i}(p_i)\supset B_{r_{i+1}}(p_{i+1})\supset \cdots
\] with $\lim_{i\to -\infty} r_i=\infty$ and $\lim_{i\to \infty} r_i=0$, such that
$\Sigma_{l}(t_i)$ converges in Hausdorff distance to $B_{r_i}(p_i)$.
\end{lem}

\begin{proof} Without loss of generality, assume $t^{*}=t_0$. By our choice of normalization,  the centroid of $\Omega_l(t_0)$ is   the origin. It thus follows from $\lim_{l\to \infty}D_l(t_0)=0$ that a subsequence of $\Sigma_l(t_0)$ converges in Hausdorff distance to $B_{r_0}(0)$ with $r_0=\sqrt{\frac{B_{\infty}(t_0)}{\pi}}$.

By (\ref{add1}),
\[
\Sigma_{l}(t_i) \supset \Sigma_{l}(t_0)\supset \Sigma_l(t_j), \quad \text{for $t_i<t_0<t_j$}.
\]
It thus follows that for each fixed $t_{i}$, the centroid of $\Sigma_l(t)$ is contained in a bounded set. Hence the conclusion follows.
\end{proof}

\begin{proof}[Proof of Theorem~\ref{t2}]
By passing to a subsequence, we shall assume that
\[
\lim_{l \to \infty} z_{i,l}=z_{i,\infty}, \quad i\geq 2,
\]
where $z_{i,\infty}$ is possibly at $\infty$. Also let $z_0=\cap_{i\in \mathbb{Z}} B_{r_i}(p_i)$ be the point given by Lemma~\ref{l3.5}. We will show that $u_l$ is uniformly bounded on any compact subset $K\subset \mathbb{C}\setminus \{z_{2,\infty}, \cdots, z_{n,\infty}, z_0\}$.
For any given $\epsilon>0$, we have thus a uniform $L_{1}\in\mathbb N$ such that for any $l>L_{1}$,
$$d(z_{i,l},K)>\epsilon.$$
Clearly, for such a compact set $K$, there exist $r_i>r_j$ such that
\[
K \subset B_{r_i}(p_i)\setminus B_{r_j}(p_j).
\]
Hence for a  $\delta>0$ small enough, we have
\[
 K \subset \mathcal{N}_{\delta}(B_{r_i}(p_i))\setminus \mathcal{N}_{\delta}(B_{r_j}(p_j)),
 \] where $\mathcal{N}_{\delta}(\cdot)$ stands for the $\delta$-neighborhood.
By  Lemma~\ref{l3.5}, there exists $L_{2}>0$ such that for $l>L_{2}$,
\[
K\subset \Sigma_l(t_j)\setminus \Sigma_l(t_i).
\]
It follows that
\[
u_l(x)\geq t_j, \quad \text{for $x\in K$ and $l>L_{2}$}.
\]

It remains to show a uniform upper bound for $u_l$. By  Lemma~\ref{l3.4}, for the chosen $\epsilon$, there exists $L_{3}\in \mathbb N$ such that for all $l>L_{3}$, any connected component $\Omega$ of $\Omega_{l}(t_{i})\setminus\Sigma_{l}(t_{i})$ satisfies
$$|\Omega|\leq\epsilon/2,\quad |\partial\Omega|\leq\epsilon/2.$$

Now for  $l>\max\{L_{1},L_{2},L_{3}\}$ large enough, any component of $\Omega_l(t_i)$ containing any singular point will not intersect $K$. Thus, if $\Sigma'$ is a connected component of $\Omega_l(t_i)$ such that   $K\cap\Sigma'\neq \emptyset$, it contains no singular point. By Lemma~\ref{l3.3}, we conclude that
\begin{equation}\label{add3}
\max_{\Sigma'}u_{l}\leq -\ln (1-{{a_{l}(t_{i})}\over {4\pi}})+t_{i}.
\end{equation}
Since $a_{l}(t_{i})\leq C_{l}$ and $\lim_{l\to\infty}C_{l}=4\pi (1+\beta_{1,\infty})<4\pi$, there exists $L_{4}\in\mathbb N$ such that for all $l>L_{4}$,
\begin{equation}\label{add4}
C_{l}\leq 4\pi +2\pi \beta_{1,\infty}<4\pi.
\end{equation}
Combining (\ref{add3}) and (\ref{add4}), we thus get a uniform upper bound for $u_{l}(z)$ for all $z\in K$ and $l>\max\{L_{1},L_{2},L_{3}, L_{4}\}$.

In summary, a subsequence of $\{u_l\}$ is uniformly bounded in any compact subset $K\subset \mathbb{C}\setminus \{z_{1,\infty}, \cdots, z_{n,\infty}, z_0\}$. In particular, by the standard $L^p$ estimates of the Poisson equation (\ref{e2.3}), we have
\[
||u_l||_{W^{2,p}(K)}\leq C, \quad \forall p>0.
\]
For $p>2$, we apply Sobolev's embedding and classical Schauder's estimate to get
\[
||u_l||_{C^{\infty}(K)}\leq C.
\]

Let $u_{\infty}:=\lim_{l\to \infty} u_l$, then $u_{\infty}$ satisfies
\[
\Delta u_{\infty}(z)=-e^{2u_{\infty}(z)}, \quad \text{for $z\in \mathbb{C}\setminus \{z_{1,\infty}, \cdots, z_{n,\infty}, z_0\}$}.
\]
It follows that the corresponding  $A^{u_{\infty}}$ and $B^{u_{\infty}}$  are just $A_{\infty}$ and $B_{\infty}$, respectively. In particular, any level set of $u_{\infty}$ has vanishing
isoperimetric deficit, which means each level set must be a round circle. In addition, (\ref{2.5}) being identity shows that $|\nabla u_{\infty}|$ are constants on the round circles $\{x; u_{\infty}(x)=c\}$.

 Hence $u_{\infty}$ has to be radially symmetric with center $z_0$, which in turn has to be the origin by our normalization (\ref{norm2}). Since $u_{\infty}$ is unbounded in view of $B_{\infty}$, $z_0=0$ has to be a singular point of $u_{\infty}$.

Henceforth, $u_{\infty}$ satisfies
\begin{align} \label{3.1}
\Delta u_{\infty}(z)=-e^{2u_{\infty}(z)}, \quad \text{for $z \in \mathbb{C}\setminus 0$}.
\end{align}
All solutions to (\ref{3.1}) are classified in ~\cite{CL2}. By direct computation, having $A_{\infty}$ and $B_{\infty}$ match with $u_{\infty}$'s, $u_{\infty}$ is necessarily given by (\ref{3.2}).

Finally, we show that
\begin{align} \label{3.3}
\lim_{l\to \infty}z_{i,l}=z_0, \quad \text{for $i\geq 2$}.
\end{align}
Suppose on the contrary, there are some singular points going to $\infty$, then there exists $L$ and $T$, such that for all $l\geq L, t\geq T$, $\Sigma_{l}(t)$ contains only parts of singular points, say $z_{i_1,l}, \cdots, z_{i_k,l}$. Let
\[
\hat{\alpha_{l}}=\beta_{i_1,l}+\cdots+\beta_{i_k,l}.
\]
Applying the analysis of Lemma~\ref{l3.1} only for the quantity $a_l(t):=\int_{\Sigma_l(t)} e^{2u_l}$, we  have
\begin{align}
a_la_l'-(2\pi\hat{\alpha_{l}})a_l'-(a_l^2-(4\pi+4\pi\hat{\alpha_{l}})a_l)\leq 0
\end{align} hold for $t\geq T$ and $l\geq L$.
Since $\hat{\alpha_{l}}>\alpha_{l}$, we deduce, by direct computation, that
\[
\limsup_{l\to \infty} a_l(t) <A_{\infty}(t), \quad t\geq T.
\]
While for any other component $\Sigma'$ of $\Omega_l(t)$, since $|\Sigma'|\to 0$ and $e^{2u_l}$ is uniformly integrable, we get
\[
\lim_{l \to \infty} \int_{\Sigma'}e^{2u_l}=0.
\]
Hence the total contribution to $A_l(t)$ does not converge to $ A_{\infty}(t)$, a contradiction.
\end{proof}

\end{document}